\documentclass[12pt,english,twoside,reqno,refpage]{amsart}
\usepackage[T1]{fontenc}
\usepackage[latin1]{inputenc}
\usepackage{amsthm}
\usepackage{amssymb}

\makeatletter
\numberwithin{equation}{section}
\numberwithin{figure}{section}
  \theoremstyle{plain}
  \newtheorem*{thm*}{\protect\theoremname}
  \theoremstyle{plain}
  \newtheorem*{cor*}{\protect\corollaryname}
\theoremstyle{plain}
\newtheorem{thm}{\protect\theoremname}[section]
  \theoremstyle{plain}
  \newtheorem{prop}[thm]{\protect\propositionname}
  \theoremstyle{plain}
  \newtheorem{lem}[thm]{\protect\lemmaname}
  \theoremstyle{plain}
  \newtheorem{cor}[thm]{\protect\corollaryname}

\makeatletter
\usepackage{a4wide}
\usepackage[all]{xy}
\usepackage{amssymb, amsmath, amsthm}
\usepackage{times}
\makeatother

\AtBeginDocument{

}

\makeatother

\usepackage{babel}
  \providecommand{\corollaryname}{Corollary}
  \providecommand{\lemmaname}{Lemma}
  \providecommand{\propositionname}{Proposition}
  \providecommand{\theoremname}{Theorem}
\providecommand{\theoremname}{Theorem}

\begin{document}

\title{A Common Axiomatic Basis for Projective Geometry and Order Geometry}

\author{Wolfram Retter}

\email{math.wolframretter(at)t-online.de}

\date{\textbf{\Large August 9, 2015}}
\begin{abstract}
A natural one-to-one correspondence between projective spaces, defined
by an axiom system published by O. Veblen and J. W. Young in 1908,
and projective join spaces, defined by an axiom system published by
M. Pieri in 1899, is presented. A projecitivity criterion for join
spaces is proved that amounts to replacing one of Pieri's projective
geometry axioms by an axiom published by G. Peano in 1889 as part
of an axiom system for order geometry. Thus, projective geometry and
order geometry have a broad common axiomatic basis. As a corollary,
it is shown how the concept of a projective join space can be derived
from the concept of a matroid. The defining properties of an equivalence
relation are used as a conceptual red thread.
\end{abstract}

\subjclass[2010]{51A05, 51D20}

\keywords{projective geometry, order geometry, line space, projective space,
projective join space, interval space, equivalence relation, matroid.}

\maketitle
\tableofcontents{}

\section{Summary}

The main results are summarized. For definitions and notation, see
the sections below.

The following theorem establishes a natural one-to-one correspondence
between the set-represented line structures and the equivalence-relational
join relations on a set. Projective spaces, defined by axioms in \cite[§1]{veblen_young_1908},
correspond to projective join spaces, defined by axioms in \cite[§1]{pieri_1899}.
\begin{thm*}
\ref{sub:correspondence_between_set_represented_line_spaces_and_equivalence_relational_join_spaces}
(correspondence between set-represented line spaces and equivalence-relational
join spaces) Let $X$ be a set.
\begin{enumerate}
\item For a set-represented line structure $L=\left(Y,\,\in\right)$ on
$X\,,$ the ternary relation $\iota\left(L\right)=\left\langle \cdot,\,\cdot,\,\cdot\right\rangle _{L}$
on $X$ defined by $\left\langle a,\, b,\, c\right\rangle _{L}:\Longleftrightarrow\left(\left(a\neq c\right)\mbox{ and }b\in\overleftrightarrow{ac}\right)$
or $\left(a=c\mbox{ and }b\in\left\{ a\right\} \right)$ is an equivalence-relational
join relation on $X\,.$
\item Vice versa, for an equivalence-relational join relation $\left\langle \cdot,\,\cdot,\,\cdot\right\rangle $
on $X\,,$ the pair\\
$\lambda\left(\left\langle \cdot,\,\cdot,\,\cdot\right\rangle \right):=\left(Y,\,\in\right)$
with $Y:=\left\{ ab|a,\, b\in X\,,\, a\neq b\right\} $ and $\in$
denoting set membership as usual is a set-represented line structure
on $X\,.$
\item $\left(\iota,\,\lambda\right)$ is an inverse pair of one-to-one correspondences
between the set-represented line structures on $X$ and the equivalence-relational
join relations on $X\,.$
\item Set-represented projective line structures on $X$ correspond to projective
join relations on $X\,.$
\end{enumerate}
\end{thm*}
\noindent The following theorem characterizes preprojective join spaces.
The equivalence (\ref{enu:join_equivalence_relationality_criterion_5})
$\Leftrightarrow$ (\ref{enu:join_equivalence_relationality_criterion_3})
amounts to replacing the projective geometry axiom Postulato XII in
\cite[§1]{pieri_1899} by Assioma XIII in \cite[§10]{peano_1889},
where it is part of an axiom system for order geometry.
\begin{thm*}
\ref{sub:join_equivalence_relationality_criterion} (join-equivalence-relationality
criterion) Let $X$ be an equivalence-relational join space. The following
conditions are equivalent:
\begin{enumerate}
\item $X$ is join-equivalence-relational.
\item $X$ is join-transitive.
\item For all $a,\, b,\, c\in X\,,$ $a\left(bc\right)\subseteq\left(ab\right)c\,.$
\item For all $a,\, b,\, c,\, d\in X\,,$ if $\left(c,\, b,\, a,\, d\right)$
is dependent, then $\left(a,\, b,\, c,\, d\right)$ is dependent. 
\item $X$ is preprojective.
\end{enumerate}
\end{thm*}
\begin{cor*}
\ref{sub:projectivity_criterion} (projectivity criterion) A join
space is projective iff it is dense and join-equivalence-relational.
\end{cor*}
The following corollary shows how the concept of a preprojective join
space can be derived from the concept of a matroid.
\begin{cor*}
\ref{sub:matroid_preprojectivity_criterion} (matroid preprojectivity
criterion) Let $X$ be a join space. $X$ is preprojective iff it
is join-transitive and the pair consisting of $X$ and the set of
join-closed sets is a matroid.
\end{cor*}
~
\begin{cor*}
\ref{sub:matroid_projectivity_criterion} (matroid projectivity criterion)
Let $X$ be a join space. $X$ is projective iff it is dense and join-transitive
and the pair consisting of $X$ and the set of join-closed sets is
a matroid.
\end{cor*}

\section{Line Spaces}

\noindent The concept of a line space is defined by some of the axioms
in \cite[§1]{veblen_young_1908}. It is proved that from each line
space, there is an isomoprhism onto a set-represented line space that
leaves all points fixed.

Let $X$ be a left vector space over a division ring $S\,,$ for example
$S=\mathbb{R}$ and $X=\mathbb{R}^{n}$ for an $n\in\mathbb{Z}_{\geq1}\,.$
Let $Y$ be the set of translated $1$-dimensional subspaces of $X$
and $*$ set membership, i.e. the relation between $X$ and $Y$ defined
by
\begin{align*}
a*y & :\Leftrightarrow a\in y\,.
\end{align*}

\noindent The pair $\left(Y,\,*\right)$ satisfies the following conditions:
\begin{itemize}
\item For $a,\, b\in X\,,$ if $a\neq b\,,$ then there is exactly one $y\in Y$
such that $a*y\,,$ $b*y\,.$ This $y$ is denoted by $\overleftrightarrow{ab}\,.$
\item For $y\in Y\,,$ there are two different $a,\, b\in X$ such that
$a*y\,,$ $b*y\,.$
\end{itemize}
\noindent A \emph{line structure} on a set $X$ is a pair $\left(Y,\,*\right)$
such that $Y$ is a set, $*$ is a binary relation between $X$ and
$Y\,,$ i.e. a subset of $X\times Y\,,$ and these conditions are
satisfied. A \emph{line space} is a triple $\left(X,\, Y,\,*\right)$
such that $X$ is a set and $\left(Y,\,*\right)$ is a line structure
on $X\,.$ The elements of $X$ are referred to as points. The elements
of $Y$ are referred to as lines. A point $x$ is said to be on a
line $y$ iff $x*y\,.$

Let $X$ be a left vector space. The line structure on $X$ defined
above is called the \emph{affine line structure} on $X\,.$ The line
space consisting of $X$ and the affine line structure on $X$ is
called the \emph{affine line space on} $X\,.$

In \cite[chapter 1, section 3]{ueberberg_2011}, a line space is called
a linear space. Here, the term 'line space' is used because the term
'linear space' is sometimes still in use as a synonym for the term
'vector space'.

Let $\left(X,\, Y,\,*\right)$ be a line space. For $a,\, b,\, c,\, d\in X\,,$

\noindent 
\begin{align*}
 & \mbox{If }a\neq b\,,\, c\neq d\mbox{ and }c,\, d\in\overleftrightarrow{ab}\,,\mbox{ then }\overleftrightarrow{cd}=\overleftrightarrow{ab}\,
\end{align*}

\noindent A line structure $\left(Y,\,*\right)$ on a set $X$ and
the line space $\left(X,\, Y,\,*\right)$ are called \emph{set-represented}
iff $Y$ is a set of subsets of $X$ and $*$ is set membership. Thus,
a set-represented line structure on $X$ is a pair $\left(Y,\,\in\right)$
such that $Y\subseteq P\left(X\right)\,,$ the power set of $X\,,$
$\in$ is set membership as usual and the following conditions are
satisfied:
\begin{itemize}
\item For $a,\, b\in X\,,$ if $a\neq b\,,$ then there is exactly one $y\in Y$
such that $a,\, b\in y\,.$
\item For $y\in Y\,,$ there are two different $a,\, b\in y\,.$
\end{itemize}
\noindent The affine line space on a left vector space is set-represented.

A \emph{strong homomorphism} between structures $\left(X_{1},\, Y_{1},\,*_{1}\right),\,\left(X_{2},\, Y_{2},\,*_{2}\right)\,,$
each consisting of two sets and a binary relation between them, is
a pair $\left(f,\, g\right)$ of maps $f:\, X_{1}\rightarrow X_{2}\,,$
$g:\, Y_{1}\rightarrow Y_{2}$ such that for all $a\in X_{1},\, y\in Y_{1}\,,$
\begin{align*}
f\left(a\right)*_{2}g\left(y\right) & \mbox{iff }a*_{1}y\,.
\end{align*}

\noindent $\left(f,\, g\right)$ is called an \emph{isomorphism} iff
$f$ and $g$ are bijections. Two such structures are called isomorphic
iff there is an isomorphism between them. In this case, if one of
them is a line space, then so is the other one. The following proposition
shows that each line space is isomorphic to a set-represented line
space with the same set of points.
\begin{prop}
\label{sub:set_representation_of_line_spaces} (set representation
of line spaces) Let $\left(X,\, Y,\,*\right)$ be a line space and
$l:\, Y\rightarrow P\left(X\right)$ the map defined by 
\begin{align*}
l\left(y\right) & :=\left\{ a\in X|a*y\right\} \,,
\end{align*}
i.e. for a line $y\,,$ $l\left(y\right)$ is the set of points on
$y\,.$ Then $\left(1_{X},\, l\right)$ is an isomorphism from $\left(X,\, Y,\,*\right)$
onto $\left(X,\, l\left(Y\right),\,\in\right)\,,$ where $\in$ denotes
set membership as usual, and $\left(X,\, L\left(Y\right),\,\in\right)$
is a set-represented line space.\end{prop}
\begin{proof}
From the assumption that $\left(X,\, Y,\,*\right)$ is a line space
it follows that it suffices to prove that $\left(1_{X},\, l\right)$
is an isomorphism from $\left(X,\, Y,\,*\right)$ onto $\left(X,\, l\left(Y\right),\,\in\right)\,.$
The map $l$ is defined in such a way that $\left(1_{X},\, l\right)$
is a stong homomorphism from $\left(X,\, Y,\,*\right)$ into $\left(X,\, P\left(X\right),\,\in\right)\,,$
i.e. for $a\in X,\, y\in Y\,,$ $1_{X}\left(a\right)\in l\left(y\right)$
iff $a*y\,.$ Furthermore, $1_{X}$ is a bijection from $X$ onto
$X\,,$ and $l$ is a map from $X$ onto $l\left(X\right)\,.$ Thus,
to prove that $\left(1_{X},\, l\right)$ is an isomorphism it suffices
to prove that $l$ is injective. For $y,\, z\in Y$ it is to be proved
that $l\left(y\right)=l\left(z\right)$ implies $y=z\,.$ There are
$a,\, b\in X$ such that $a\neq b\,,$ $a*y$ and $b*y\,,$ i.e. $a,\, b\in l\left(y\right)\,.$
With the assumption $l\left(y\right)=l\left(z\right)$ it follows
that $a,\, b\in l\left(z\right)\,,$ i.e. $a*z$ and $b*z\,.$ With
$a\neq b\,,$ $a*y$ and $b*y$ it follows that $z=y\,.$ 
\end{proof}
In view of \ref{sub:set_representation_of_line_spaces} (set representation
of line spaces), the abstract theory of line spaces can be restricted
to set-represented line spaces. The more general concept is in use
for several reasons. Some line spaces are, in their most convenient
defintion, not set-represented. For example, the the projective line
space over a left vector space as defined below is not set-represented.
Also, the concept of a line space is a particular case of a concept
from general incidence geometry as in \cite[chapter 1, section 3]{ueberberg_2011}.

\section{Join Spaces}

\noindent The concept of a projective join space is built incrementally
via the concept of an equivalence-relational join space. Then all
axioms defining this concept are listed with their numbering in \cite[§1]{pieri_1899}
and in \cite[chapter II, sect. 4]{whitehead_1906}, where a translation
of the axioms into English has been given. Finally, a one-to-one correspondence
between line spaces and equivalence-relational join spaces is presented.
Projective line spaces correspond to projective join spaces.

Let $X$ be a left vector space over a division ring $S\,,$ for example
$S=\mathbb{R}$ and $X=\mathbb{R}^{n}$ for an $n\in\mathbb{Z}_{\geq1}\,.$
Let $\left\langle \cdot,\,\cdot,\,\cdot\right\rangle $ be the ternary
relation on $X$ defined by
\begin{align*}
\left\langle x,\, y,\, z\right\rangle  & :\Leftrightarrow\mbox{There is a }\lambda\in S\mbox{ such that }y=x+\lambda\left(z-x\right)\,.
\end{align*}

\noindent $X$ together with this ternary relation satisfies the following
conditions:
\begin{itemize}
\item For $a\in X\,,$ the binary relation $\left\langle a,\,\cdot,\,\cdot\right\rangle $
is reflexive on $X\,.$
\item For $b\in X\,,$ the binary relation $\left\langle \cdot,\, b,\,\cdot\right\rangle $
is symmetric.
\item For $x,\, y\in X\,,$ $\left\langle x,\, y,\, x\right\rangle $ implies
$y=x\,.$
\end{itemize}
\noindent A \emph{join} \emph{relation} on a set $X$ is a ternary
relation $\left\langle \cdot,\,\cdot,\,\cdot\right\rangle $ on $X$
such that these conditions are satisfied. A \emph{join space} is a
pair consisting of a set $X$ and a join relation on $X\,.$ 

Let $X$ be a left vector space. The join relation on $X$ defined
above is called the \emph{affine join relation} on $X\,.$ The join
space consisting of $X$ and the affine join relation on $X$ is called
the \emph{affine join space on} $X\,.$

Let $X$ be join space. The ternary relation $\left\langle \cdot,\,\cdot,\,\cdot\right\rangle $
is determined by the family of sets $\left(ac\right)_{a,\, c\in X}$
defined by

\noindent 
\begin{align*}
ac & :=\left\langle a,\,\cdot,\, c\right\rangle \\
 & =\left\{ x\in X\,|\,\left\langle a,\, x,\, c\right\rangle \right\} \,.
\end{align*}

\noindent The set $ac$ is called the \emph{join} of $a$ and $c\,.$
The above set of conditons defining the concepts of a join relation
and a join space is equivalent to the following set of conditions:
\begin{itemize}
\item For $a,\, c\in X\,,$ if $a\neq c\,,$ then $a\in ac\,.$
\item For $a,\, c\in X\,,$ if $a\neq c\,,$ then $ac\subseteq ca\,.$
\item For $a\in X\,,$ $aa=\left\{ a\right\} \,.$
\end{itemize}
\noindent A join space $\left(X,\,\left\langle \cdot,\,\cdot,\,\cdot\right\rangle \right)$
is also simply denoted by $X$ when it is clear from the context whether
the join space or only the set is meant.

In \cite[chapter I, 7.6]{vel_1993}, the term 'join space' is used
only for a particular case. Here, the terms 'join' and 'join space'
replace the terms 'interval' and 'interval space' in \cite[chapter I, 2.2]{verheul_1993},
\cite{retter_2013} and \cite{retter_2014} because the terms 'interval'
and 'interval space' suggest a a narrower class of join spaces, for
which the following example is typical: Let $X$ be a left vector
space over a totally ordered division ring $S\,,$ for example $S=\mathbb{R}$
and $X=\mathbb{R}^{n}$ for an $n\in\mathbb{Z}_{\geq1}\,.$ The ternary
relation $\left\langle \cdot,\,\cdot,\,\cdot\right\rangle $ on $X$
defined by

\noindent 
\begin{align*}
\left\langle x,\, y,\, z\right\rangle  & :\Leftrightarrow\mbox{There is a }\lambda\in S\mbox{ such that }0\leq\lambda\leq1\mbox{ and }y=x+\lambda\left(z-x\right)\,.
\end{align*}

\noindent is a join relation on $X\,.$ It is called the \emph{line-segment
 join relation} on $X\,.$ The join space consisting of $X$ and the
line-segment join relation on $X$ is called the \emph{line-segment
join space on} $X\,.$ 

In \cite[chapter I, 4.1]{vel_1993}, the term 'interval space' has
been used even in a wider sense than in \cite[chapter I, 2.2]{verheul_1993},
\cite{retter_2013} and \cite{retter_2014}.

Let $X$ be a join space.

For $a,\, b,\, c\in X\,,$
\begin{align*}
\left\langle a,\, b,\, c\right\rangle ,\, a\neq b & \Longrightarrow a\neq c\,.
\end{align*}

For $A,\, B,\, C\subseteq X\,,$
\begin{align*}
\left\langle A,\, B,\, C\right\rangle  & :\Longleftrightarrow\mbox{There are }a\in A,\, b\in B,\, c\in C\mbox{ such that }\left\langle a,\, b,\, c\right\rangle \,.
\end{align*}

\noindent \begin{flushleft}
In this notation, when $A\,,$ $B$ or $C$ is a singleton $\left\{ x\right\} \,,$
it may be replaced by $x\,.$
\par\end{flushleft}

\begin{flushleft}
For $A,\, C\subseteq X\,,$ the \emph{join} of $A$ and $C$ is the
set
\par\end{flushleft}

\noindent 
\begin{align*}
AC & :=\left\{ x\in X\,|\,\left\langle A,\, x,\, C\right\rangle \right\} \,.
\end{align*}

\noindent \begin{flushleft}
In this notation, when $A$ or $C$ is a singleton $\left\{ x\right\} \,,$
it may be replaced by $x\,.$
\par\end{flushleft}

\noindent Part (\ref{enu:set_join_operator_1}) of the following proposition
is cited from \cite[Theorem 2.3]{prenowitz_jantosciak_1979}. Parts
(\ref{enu:set_join_operator_3}) and (\ref{enu:set_join_operator_4})
are cited from \cite[Theorem 2.1]{prenowitz_jantosciak_1979}.
\begin{prop}
\label{sub:set_join_operator} (set join operator) Let $X$ be a join
space.
\begin{enumerate}
\item \label{enu:set_join_operator_1}For $A,\, B\subseteq X\,,$ $AB=BA\,.$
\item \label{enu:set_join_operator_2}For $A,\, B\subseteq X\,,$ if $B\neq\emptyset\,,$
then $A\subseteq AB\,.$
\item \label{enu:set_join_operator_3}For $A,\, B,\, C\subseteq X\,,$ $A\subseteq B\Longrightarrow AC\subseteq B,C\,.$
\item \label{enu:set_join_operator_4}For $A,\, B,\, C\subseteq X\,,$ $A\subseteq B\Longrightarrow\left[C,\, A\right]\subseteq\left[C,\, B\right]\,.$
\item \label{enu:set_join_operator_5}For $A,\, B,\, C,\, D\subseteq X\,,$
$A\subseteq B\mbox{ and }C\subseteq D\Longrightarrow AC\subseteq BD\,.$
\end{enumerate}
\end{prop}
\begin{proof}
~
\begin{enumerate}
\item \cite[Theorem 2.3]{prenowitz_jantosciak_1979}
\item The assumption $B\neq\emptyset$ says that there is a $b\in B\,.$
For each $a\in A\,,$ $a\in ab\,.$ Substituting $b\in B\,,$ $a\in aB\,.$
\item \cite[Theorem 2.1]{prenowitz_jantosciak_1979}
\item \cite[Theorem 2.1]{prenowitz_jantosciak_1979}
\item follows from (\ref{enu:set_join_operator_4}) and (\ref{enu:set_join_operator_5})
\end{enumerate}
\end{proof}
\noindent Let $\left(X,\,\left\langle \cdot,\,\cdot,\,\cdot\right\rangle \right)$
be a join space and $A\subseteq X\,.$

The binary relation $\left\langle A,\,\cdot,\,\cdot\right\rangle $
is reflexive on $X\,.$

The following definitions of transitivity and symmetry with respect
to $A$ can be summarized as follows: For a property $P$ of a binary
relation, $X$ is said to have property $P$ with respect to $A$
iff the binary relation $\left\langle A,\,\cdot,\,\cdot\right\rangle $
restricted to $X\setminus A$ has property $P\,.$

$\left(X,\,\left\langle \cdot,\,\cdot,\,\cdot\right\rangle \right)$
and $\left\langle \cdot,\,\cdot,\,\cdot\right\rangle $ are called
\emph{transitive with respect to} $A$ or $A$-transitive iff the
following equivalent conditions are satisfied:
\begin{itemize}
\item The binary relation $\left\langle A,\,\cdot,\,\cdot\right\rangle $
is transitive on $X\setminus A\,.$
\item For all $b,\, c\in X\,,$ if $\left\langle A,\, b,\, c\right\rangle $
and $b,\, c\notin A\,,$ then $Ab\subseteq Ac\,.$
\item For all $b,\, c\in X\,,$ if $\left\langle A,\, b,\, c\right\rangle $
and $b\notin A\,,$ then $Ab\subseteq Ac\,.$
\item For all $b,\, c\in X\,,$ if $\left\langle A,\, b,\, c\right\rangle \,,$
then $Ab\subseteq Ac\,.$
\item The binary relation $\left\langle A,\,\cdot,\,\cdot\right\rangle $
is transitive on $X\,.$
\end{itemize}
\noindent The affine join space on a left vector space is transitive
with respect to each subset.. The line-segment join space on a left
vector space over a totally ordered division ring is transitive with
respect to each subset.

$\left(X,\,\left\langle \cdot,\,\cdot,\,\cdot\right\rangle \right)$
and $\left\langle \cdot,\,\cdot,\,\cdot\right\rangle $ are called\emph{
symmetric with respect to} $A$ or $A$-symmetric iff the following
equivalent conditions are satisfied:
\begin{itemize}
\item The binary relation $\left\langle A,\,\cdot,\,\cdot\right\rangle $
is symmetric on $X\setminus A\,.$
\item For all $b,\, c\in X\,,$ $c\in Ab$ and $c\notin A$ implies $b\in Ac\,.$
\end{itemize}
\noindent The affine join space on a left vector space is symmetric
with respect to each subset.. The line-segment join space on a left
vector space over a totally ordered division ring is in general not
symmetric with respect to a subset.

When $X$ is $A$-symmetric, in general the binary relation $\left\langle A,\,\cdot,\,\cdot\right\rangle $
won't be symmetric on all of $X\,.$ This situation is different from
the case that $X$ is $A$-transitive, where the binary relation$\left\langle A,\,\cdot,\,\cdot\right\rangle $
is transitive on all of $X\,.$

$\left(X,\,\left\langle \cdot,\,\cdot,\,\cdot\right\rangle \right)$
and $\left\langle \cdot,\,\cdot,\,\cdot\right\rangle $ are called\emph{
equivalence-relational with respect to} $A$ or $A$-equivalence-relational
 iff the following equivalent conditions are satisfiesd:
\begin{itemize}
\item The restriciton of the binary relation $\left\langle a,\,\cdot,\,\cdot\right\rangle $
to $X\setminus A$ is an equivalence relation.
\item $X$ is $A$-transitive and $A$-symmetric.
\end{itemize}
\noindent The affine join space on a left vector space is equivalence-relational
with respect to each subset.. The line-segment join space on a left
vector space over a totally ordered division ring is in general not
equivalence-relational with respect to a subset.

In the definitions of transitivity, symmetry and euqivalence-relationality
with respect to $A\,,$ when $A$ is a singleton $\left\{ a\right\} \,,$
it may be replaced by $a\,.$ Thus, $X$ is $a$-transitive, $a$-symmetric,
$a$-equivalence-relational iff it is $\left\{ a\right\} $-transitive,
$\left\{ a\right\} $-symmetric, $\left\{ a\right\} $-equivalence-relational,
respectively.
\begin{lem}
\label{sub:transitivity_relative_to_a_set} (transitivity relative
to a set) Let $X$ be a join space. For $A\subseteq X\,,$ $X$ is
$A$-transitive iff for all $b\in X\,,$ $A\left(Ab\right)\subseteq Ab\,.$\end{lem}
\begin{proof}
The following conditions are equivalent:

\begin{align*}
 & \mbox{For all }b\in X\,,\, A\left(Ab\right)\subseteq Ab\,.\\
 & \mbox{For all }b,\, x,\, y\in X\,,\, x\in Ay\mbox{ and }y\in Ab\Longrightarrow x\in Ab\,.\\
 & \mbox{For all }x,\, y\in A\,,\,\left\langle A,\, x,\, y\right\rangle \mbox{ and }\left\langle A,\, y,\, b\right\rangle \Longrightarrow\left\langle A,\, x,\, b\right\rangle \,.\\
 & X\mbox{ is }A\mbox{-transitive.}
\end{align*}

\end{proof}
Let $X$ be a join space.

A triple $\left(A,\, B,\, C\right)$ of subsets of $X$ is called
\emph{dependent} iff $A\cap B\neq\emptyset$ or $\left(AB\right)\cap C\neq\emptyset\,.$
In this definition, when $A\,,$ $B$ or $C$ is a singleton $\left\{ x\right\} \,,$
it may be replaced by $x\,.$ For example, a triple $\left(a,\, b,\, c\right)$
of elements of $X$ is dependent iff $b=a$ or $c\in ab\,.$

A quadruple $\left(A,\, B,\, C,\, D\right)$ of subsets of $X$ is
called \emph{dependent} iff the following equivalent conditions are
satisfied:
\begin{itemize}
\item $\left(A,\, B,\, C\right)$ is dependent or $\left(AB\right)C\cap D\neq\emptyset\,.$
\item $A\cap B\neq\emptyset$ or $\left(AB\right)\cap C\neq\emptyset$ or
$\left(AB\right)C\cap D\neq\emptyset\,.$
\end{itemize}
\noindent In this definition, when $A\,,$ $B\,,$ $C$ or $D$ is
a singleton $\left\{ x\right\} \,,$ it may be replaced by $x\,.$
For example, a quadruple $\left(a,\, b,\, c,\, d\right)$ of elements
of $X$ is dependent iff $b=a$ or $c\in ab$ or $d\in\left(ab\right)c\,.$
Each of the following two figures illustrates the case $d\in\left(ab\right)c\,.$ 

~

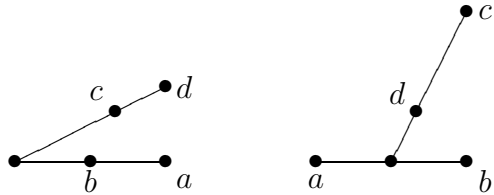
\noindent 
\[
\xy<1cm,0cm>:(2,0)*=0{\bullet}="a",(2.25,-0.25)*=0{a},(1,0)*=0{\bullet}="b",(1,-0.25)*=0{b},(2,1)*=0{\bullet}="d",(2.25,1)*=0{d},(0,0)*=0{\bullet}="e",(2,2)*=0{}="f","e";"d"**@{-}?!{"f";"b"}*{\bullet}="c",(1.08,0.92)*=0{c},"e";"a"**@{-},(4,0)*=0{\bullet}="a1",(4,-0.25)*=0{a},(6,0)*=0{\bullet}="b1",(6.25,-0.25)*=0{b},(6,2)*=0{\bullet}="c1",(6.25,2)*=0{c},(6,1)*=0{}="a'",(5,0)*=0{\bullet}="c'","a1";"a'"**@{}?!{"c1";"c'"}*{\bullet}="d1",(5.08,0.92)*=0{d},"a1";"b1"**@{-},"c1";"c'"**@{-}\endxy
\]
~

\noindent For the next proposition, each of the following two figures
illustrates condition (\ref{enu:rejoinability_criterion_1}).

~

\noindent 
\[
\xy<1cm,0cm>:(2,0)*=0{\bullet}="a",(2.25,-0.25)*=0{a},(1,0)*=0{\bullet}="b",(1,-0.25)*=0{b},(2,1)*=0{\bullet}="d",(2.25,1)*=0{d},(0,0)*=0{\bullet}="e",(2,2)*=0{\bullet}="f","e";"d"**@{.}?!{"f";"b"}*{\bullet}="c",(1.08,0.92)*=0{c},"e";"a"**@{.},"f";"a"**@{-},"f";"b"**@{-},(4,0)*=0{\bullet}="a1",(4,-0.25)*=0{a},(6,0)*=0{\bullet}="b1",(6.25,-0.25)*=0{b},(6,2)*=0{\bullet}="c1",(6.25,2)*=0{c},(6,1)*=0{\bullet}="a'",(5,0)*=0{\bullet}="c'","a1";"a'"**@{-}?!{"c1";"c'"}*{\bullet}="d1",(5.08,0.92)*=0{d},"a1";"b1"**@{.},"c1";"b1"**@{-},"c1";"c'"**@{.}\endxy
\]
~
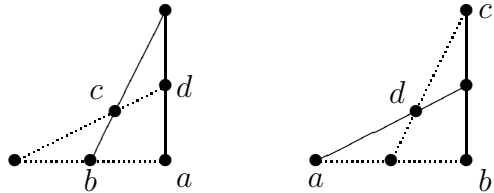
\begin{prop}
\label{sub:rejoinability_criterion} (rejoinability criterion) Let
$X$ be a join space. For $a,\, b,\, c,\, d\in X\,,$ if $X$ is $a$-equivalence-relational
and $b$-symmetric, then the following conditions are equivalent:
\begin{enumerate}
\item \label{enu:rejoinability_criterion_1}If $d\in a\left(bc\right)\,,$
then $d\in\left(ab\right)c\,.$
\item \label{enu:rejoinability_criterion_2}If $\left(c,\, b,\, a,\, d\right)$
is dependent, then $\left(a,\, b,\, c,\, d\right)$ is dependent.
\end{enumerate}
\end{prop}
\begin{proof}
The assumption that $X$ is $a$-equivalence-relational entails that
$X$ is $a$-transitive and $a$-symmetric.

Step 1. (\ref{enu:rejoinability_criterion_1}) $\Rightarrow$ (\ref{enu:rejoinability_criterion_2}).
(\ref{enu:rejoinability_criterion_2}) says: If $b=c$ or $a\in bc$
or $d\in a\left(bc\right)\,,$ then $b=a$ or $c\in ab$ or $d\in(ab)c\,.$
From (\ref{enu:rejoinability_criterion_1}) it follows that it suffices
to prove that $b=c$ or $a\in bc$ implies $b=a$ or $c\in ab\,.$

Case 1.1. $b=c\,.$ Substituting into $b\in ab\,,$ $c\in ab\,.$

Case 1.2. $a\in bc\,,$ i.e. $\left\langle b,\, a,\, c\right\rangle \,.$
It is to be proved that $b\neq a$ implies $c\in ab\,,$ i.e. $\left\langle b,\, c,\, a\right\rangle \,.$
This claim follows from the assumptions that $\left\langle b,\, a,\, c\right\rangle \,,$
$b\neq a$ and $X$ is $b$-symmetric.

Step 2. (\ref{enu:rejoinability_criterion_2}) $\Rightarrow$ (\ref{enu:rejoinability_criterion_1}).
(\ref{enu:rejoinability_criterion_2}) says: If $b=c$ or $a\in bc$
or $d\in a\left(bc\right)\,,$ then $b=a$ or $c\in ab$ or $d\in(ab)c\,.$
With the assumption $d\in a\left(bc\right)$ it follows that $b=a$
or $c\in ab$ or $d\in\left(ab\right)c\,.$

\noindent Case 2.1. $b=a\,.$ It follows from this assumption, from
the assumption that $X$ is $a$-transitive by \ref{sub:transitivity_relative_to_a_set}
(transitivity relative to a set) and from $\left\{ a\right\} \subseteq ab$
by \ref{sub:set_join_operator} (\ref{enu:set_join_operator_3}) (set
join operator): 
\begin{align*}
a\left(bc\right) & =a\left(ac\right)\\
 & \subseteq ac\\
 & \subseteq\left(ab\right)c\,.
\end{align*}

\noindent With the assumption that $d\in a\left(bc\right)\,,$ it
follows that $d\in\left(ab\right)c\,.$

Case 2.2. $c\in ab\,,$ i.e. $\left\langle b,\, c,\, a\right\rangle \,.$
From this assumption and the assumption that $X$ is $b$-transitive
it follows  that $bc\subseteq ba\,.$ It follows by \ref{sub:set_join_operator}
(\ref{enu:set_join_operator_4}) (set join operator), with the assumption
that $X$ is $a$-transitive by \ref{sub:transitivity_relative_to_a_set}
(transitivity relative to a set) and by \ref{sub:set_join_operator}
(\ref{enu:set_join_operator_2}) (set join operator): 
\begin{align*}
a\left(bc\right) & \subseteq a\left(ba\right)\\
 & =a\left(ab\right)\\
 & \subseteq ab\\
 & \subseteq\left(ab\right)c\,.
\end{align*}
With the assumption that $d\in a\left(bc\right)$ it follows that
$d\in\left(ab\right)c\,.$

Case 2.3. $d\in\left(ab\right)c\,.$ For this case, nothing is to
be proved.\end{proof}
\begin{prop}
\label{sub:join_spaces_symmetric_with_respect_to_a_base_set} (join
spaces symmetric with respect to a base-set) Let $X$ be a join space.
For $A,\, B,\, C\subseteq X\,,$ if $X$ is $A$-symmetric, $C\neq\emptyset$
and $\left(A,\, B,\, C\right)$ is dependent, then $\left(A,\, C,\, B\right)$
is dependent. \end{prop}
\begin{proof}
The assumption that $\left(A,\, B,\, C\right)$ is dependent says
that $A\cap B\neq\emptyset$ or $\left(AB\right)\cap C\neq\emptyset\,.$
It is to be proved that $A\cap C\neq\emptyset$ or $\left(AC\right)\cap B\neq\emptyset\,.$

Case 1.$A\cap B\neq\emptyset\,.$ From the assumption that $C\neq\emptyset$
it follows by \ref{sub:set_join_operator} (\ref{enu:set_join_operator_2})
(set join operator) that $A\subseteq AC\,.$ Therefore, $A\cap B\subseteq\left(AC\right)\cap B\,.$
With the assumption that $A\cap B\neq\emptyset$ it follows that $\left(AC\right)\cap B\neq\emptyset\,.$

Case 2. $\left(AB\right)\cap C\neq\emptyset\,,$ i.e. there are $b\in B\,,$
$c\in C$ such that $c\in Ab\,.$

Case 2.1. $c\notin A\,.$ From this assumption and the assumptions
that $c\in Ab$ and $X$ is $A$-symmetric it follows that $b\in Ac\,.$
With  the assumptions $b\in B\,,$ $c\in C$ it follows that $b\in AC\cap B\,.$
Consequently, $\left(AC\right)\cap B\neq\emptyset\,.$

Case 2.2. $c\in A\,.$ With the assumption $c\in C\,,$ $c\in A\cap C\,.$
Consequently, $A\cap C\neq\emptyset\,.$
\end{proof}
\noindent For the next proposition, each of the following two figures
illustrates
\begin{itemize}
\item the case $d\in\left(ab\right)c$ of condition (\ref{enu:quadruple_dependence_criterion_1})
\item the case $\left(ab\right)\cap\left(cd\right)\neq\emptyset$ of condition
(\ref{enu:quadruple_dependence_criterion_2})
\end{itemize}
~

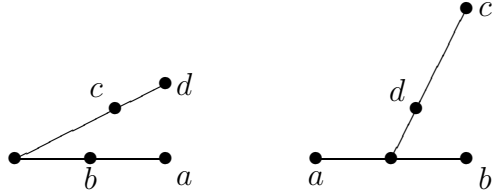
\noindent 
\[
\xy<1cm,0cm>:(2,0)*=0{\bullet}="a",(2.25,-0.25)*=0{a},(1,0)*=0{\bullet}="b",(1,-0.25)*=0{b},(2,1)*=0{\bullet}="d",(2.25,1)*=0{d},(0,0)*=0{\bullet}="e",(2,2)*=0{}="f","e";"d"**@{-}?!{"f";"b"}*{\bullet}="c",(1.08,0.92)*=0{c},"e";"a"**@{-},(4,0)*=0{\bullet}="a1",(4,-0.25)*=0{a},(6,0)*=0{\bullet}="b1",(6.25,-0.25)*=0{b},(6,2)*=0{\bullet}="c1",(6.25,2)*=0{c},(6,1)*=0{}="a'",(5,0)*=0{\bullet}="c'","a1";"a'"**@{}?!{"c1";"c'"}*{\bullet}="d1",(5.08,0.92)*=0{d},"a1";"b1"**@{-},"c1";"c'"**@{-}\endxy
\]
~
\begin{prop}
\label{sub:quadruple_dependence_criterion} (quadruple dependence
criterion) Let $X$ be a join space. For $a,\, b,\, c,\, d\in X\,,$
if $X$ is $c$-symmetric, then the following conditions are equivalent:
\begin{enumerate}
\item \label{enu:quadruple_dependence_criterion_1}$\left(a,\, b,\, c,\, d\right)$
is dependent.
\item \label{enu:quadruple_dependence_criterion_2}$a=b$ or $c=d$ or $\left(ab\right)\cap\left(cd\right)\neq\emptyset\,.$
\end{enumerate}
\end{prop}
\begin{proof}
From the assumption that $X$ is $c$-symmetric it follows by \ref{sub:join_spaces_symmetric_with_respect_to_a_base_set}
(join spaces symmetric with respect to a base-set) that the following
conditions are equivalent:

\begin{align*}
 & \left(a,\, b,\, c,\, d\right)\mbox{ is dependent.}\\
 & a=b\mbox{ or }c\in ab\mbox{ or }d\in\left(ab\right)c\,.\\
 & a=b\mbox{ or }\left\{ c\right\} \cap ab\neq\emptyset\mbox{ or }d\in\left(ab\right)c\,.\\
 & a=b\mbox{ or }\left(c,\, ab,\, d\right)\mbox{ is dependent.}\\
 & a=b\mbox{ or }\left(c,\, d,\, ab\right)\mbox{ is dependent.}\\
 & a=b\mbox{ or }c=d\mbox{ or }\left(ab\right)\cap\left(cd\right)\neq\emptyset\,.
\end{align*}
\end{proof}
\begin{prop}
\label{sub:base_point_equivalence_relationality_criterion} (base-point
equivalence-relationality criterion) Let $X$ be a join space. For
$a\in X\,,$ the following conditions are equivalent:
\begin{enumerate}
\item \label{enu:base_point_equivalence_relationality_criterion_1}$X$
is $a$-equivalence-relational.
\item \label{enu:base_point_equivalence_relationality_criterion_2}For all
$b,\, c\in X\,,$ $\left\langle a,\, b,\, c\right\rangle $ and $a\neq b$
imply $ab\subseteq ac$ and $\left\langle a,\, c,\, b\right\rangle \,.$
\item \label{enu:base_point_equivalence_relationality_criterion_3}For all
$b,\, c\in X\,,$ $\left\langle a,\, b,\, c\right\rangle $ and $a\neq b$
imply $ab=ac\,.$
\end{enumerate}
\end{prop}
\begin{proof}
Step 1. (\ref{enu:base_point_equivalence_relationality_criterion_1})
$\Leftrightarrow$ (\ref{enu:base_point_equivalence_relationality_criterion_2}).
(\ref{enu:base_point_equivalence_relationality_criterion_1}) says
that $X$ is $a$-transitive and $a$-symmetric. This condition is
equivalent to (\ref{enu:base_point_equivalence_relationality_criterion_2}).

Step 2. (\ref{enu:base_point_equivalence_relationality_criterion_2})
$\Rightarrow$ (\ref{enu:base_point_equivalence_relationality_criterion_3}).
It remains to be proved that $\left\langle a,\, b,\, c\right\rangle $
and $a\neq b$ imply $ac\subseteq ab\,.$ From the assumptions $\left\langle a,\, b,\, c\right\rangle \,,$
$a\neq b$ and (\ref{enu:base_point_equivalence_relationality_criterion_2})
it follows that $\left\langle a,\, c,\, b\right\rangle $ and $a\neq c\,.$
With (\ref{enu:base_point_equivalence_relationality_criterion_2})
again, $ac\subseteq ab\,.$

Step 3. (\ref{enu:base_point_equivalence_relationality_criterion_3})
$\Rightarrow$ (\ref{enu:base_point_equivalence_relationality_criterion_2}).
It remains to be proved that $\left\langle a,\, b,\, c\right\rangle $
and $a\neq b$ imply $\left\langle a,\, c,\, b\right\rangle \,.$
With (\ref{enu:base_point_equivalence_relationality_criterion_3}),
$ab=ac\,.$ Substituting into $c\in ac\,,$ $c\in ab\,,$ i.e. $\left\langle a,\, c,\, b\right\rangle \,.$
\end{proof}
\noindent Let $\left(X,\,\left\langle \cdot,\,\cdot,\,\cdot\right\rangle \right)$
be a join space.

For $a\in X\,,$ the binary relation $\left\langle a,\,\cdot,\,\cdot\right\rangle $
is reflexive on $X\,.$ The following definitions of transitivity
and symmetry can be summarized as follows: For a property $P$ of
a binary relation, $X$ is said to have property $P$ iff for all
$a\in X\,,$ the binary relation $\left\langle a,\,\cdot,\,\cdot\right\rangle $
restricted to $X\setminus\left\{ a\right\} $ has property $P\,.$

$\left(X,\,\left\langle \cdot,\,\cdot,\,\cdot\right\rangle \right)$
and $\left\langle \cdot,\,\cdot,\,\cdot\right\rangle $ are called\emph{
transitive} iff the following equivalent conditions are satisfied,
where the last condition is equivalent to the others by \ref{sub:transitivity_relative_to_a_set}
(transitivity relative to a set).
\begin{itemize}
\item For each $a\in X\,,$ $X$ is $a$-transitive.
\item For all $b,\, c\in X\,,$ if $\left\langle a,\, b,\, c\right\rangle \,,$
$a\neq b$ and $a\neq c\,,$ then $ab\subseteq ac\,.$
\item For all $a,\, b,\, c\in X\,,$ if $\left\langle a,\, b,\, c\right\rangle $
and $a\neq b\,,$ then $ab\subseteq ac\,.$
\item For all $a,\, b,\, c\in X\,,$ if $\left\langle a,\, b,\, c\right\rangle \,,$
then $ab\subseteq ac\,.$
\item For each $a\in X\,,$ the binary relation $\left\langle a,\,\cdot,\,\cdot\right\rangle $
is transitive on $X\,.$
\item For all $a,\, b\in X\,,$ $a\left(ab\right)\subseteq ab\,.$
\end{itemize}
\noindent The affine join space on a left vector space is transitive.
The line-segment join space on a left vector space over a totally
ordered division ring is transitive.

$\left(X,\,\left\langle \cdot,\,\cdot,\,\cdot\right\rangle \right)$
and $\left\langle \cdot,\,\cdot,\,\cdot\right\rangle $ are called\emph{
symmetric} iff the following equivalent conditions are satisfied:
\begin{itemize}
\item For each $a\in X\,,$ $X$ is $a$-symmetric.
\item For all $a,\, b,\, c\in X\,,$ $\left\langle a,\, b,\, c\right\rangle $
and $a\neq b$ implies $\left\langle a,\, c,\, b\right\rangle \,.$
\item For each $A\subseteq X\,,$ $X$ is $A$-symmetric.
\end{itemize}
\noindent The affine join space on a left vector space is symmetric.
The line-segment join space on a left vector space over a totally
ordered division ring is in general not symmetric.

When $X$ is symmetric, in general the binary relations $\left\langle a,\,\cdot,\,\cdot\right\rangle $
won't be symmetric on all of $X\,.$ This situation is different from
the case that $X$ is transitive, where the binary relations $\left\langle a,\,\cdot,\,\cdot\right\rangle $
are transitive on all of $X\,.$

On the other hand, when $X$ is transitive, in general it won't be
$A$-transitive for all $A\subseteq X\,.$ This situation is different
from the case that $X$ is symmetric, where $X$ is $A$-symmetric
for all $A\subseteq X\,.$

$\left(X,\,\left\langle \cdot,\,\cdot,\,\cdot\right\rangle \right)$
and $\left\langle \cdot,\,\cdot,\,\cdot\right\rangle $ are called\emph{
equivalence-relational} iff the following equivalent conditions are
satisfied:
\begin{itemize}
\item For each $a\in X\,,$ $X$ is $a$-equivalence-relational.
\item $X$ is transitive and symmetric.
\end{itemize}
\noindent The affine join space on a left vector space is equivalence-relational.
The line-segment join space on a left vector space over a totally
ordered division ring is in general not equivalence-relational.
\begin{prop}
\label{sub:join_space_associated_with_a_line_space} (join space associated
with a line space) Let $\left(X,\, Y,\,*\right)$ be a line space.
Then the ternary relation $\left\langle \cdot,\,\cdot,\,\cdot\right\rangle $
on $X$ defined by $\left\langle a,\, b,\, c\right\rangle :\Longleftrightarrow\left(\left(a\neq c\right)\mbox{ and }b\in\overleftrightarrow{ac}\right)$
or $\left(a=c\mbox{ and }b\in\left\{ a\right\} \right)$ is an equivalence-relational
join relation on $X\,.$\end{prop}
\begin{proof}
Step 1. Proof that $\left\langle \cdot,\,\cdot,\,\cdot\right\rangle $
is a join relation on $X\,.$

Step 1.1. For $a,\, c\in X\,,$ if $a\neq c\,,$ then $c\in\overleftrightarrow{ac}\,,$
i.e. $\left\langle a,\, c,\, c\right\rangle \,,$ and $\overleftrightarrow{ac}=\overleftrightarrow{ca}\,,$
therefore, for $b\in X\,,$ $\left\langle a,\, b,\, c\right\rangle $
implies $\left\langle c,\, b,\, a\right\rangle \,.$ And for $a,\, b\in X\,,$
if $\left\langle a,\, b,\, a\right\rangle \,,$ then $b\in\left\{ a\right\} \,,$
i.e. $b=a\,.$ Consequently, $\left\langle \cdot,\,\cdot,\,\cdot\right\rangle $
is a join relation on $X\,.$

Step 2. Proof that $\left\langle \cdot,\,\cdot,\,\cdot\right\rangle $
is equivalence-relational, i.e. for $a\in X\,,$ $\left\langle \cdot,\,\cdot,\,\cdot\right\rangle $
is $a$-equivalence-relational.

Step 2.1. Proof that $\left\langle \cdot,\,\cdot,\,\cdot\right\rangle $
is $a$-symmetric, i.e. for $y,\, z\in X\setminus\left\{ a\right\} \,,$
$\left\langle a,\, y,\, z\right\rangle $implies $\left\langle a,\, z,\, y\right\rangle \,,$
i.e. for $y,\, z\in X\,,$ from $y\neq a\,,$ $z\neq a$ and $y\in\overleftrightarrow{az}$
it follows that $z\in\overleftrightarrow{ay}\,.$ From $a\in\overleftrightarrow{az}$
and the assumptions $y\in\overleftrightarrow{az}$ and $a\neq y$
it follows that $\overleftrightarrow{ay}=\overleftrightarrow{az}\,.$
Substituting into $z\in\overleftrightarrow{az}\,,$ $z\in\overleftrightarrow{ay}\,.$

Step 2.2. Proof that $\left\langle \cdot,\,\cdot,\,\cdot\right\rangle $
is $a$-transitive, i.e. for $x,\, y,\, z\in X\setminus\left\{ a\right\} \,,$
$\left\langle a,\, x,\, y\right\rangle $ and $\left\langle a,\, y,\, z\right\rangle $
implies $\left\langle a,\, x,\, z\right\rangle \,,$ i.e. from $x\neq a\,,$
$y\neq a\,,$ $z\neq a\,,$ $x\in\overleftrightarrow{ay}$ and $y\in\overleftrightarrow{az}$
it follows that $x\in\overleftrightarrow{az}\,.$ From $a\in\overleftrightarrow{az}$
and the assumptions $y\in\overleftrightarrow{az}$ and $y\neq a$
it follows that $\overleftrightarrow{ay}=\overleftrightarrow{az}\,.$
Substituting into the assumption $x\in\overleftrightarrow{ay}\,,$
$x\in\overleftrightarrow{az}\,.$
\end{proof}
\noindent For a line structure $\left(Y,\,*\right)$ on a set $X\,,$
the equivalence-relational join relation $\left\langle \cdot,\,\cdot,\,\cdot\right\rangle $
on $X$ as defined in \ref{sub:join_space_associated_with_a_line_space}
(join space associated with a line space) is called the join relation
associated with $\left(Y,\,*\right)\,.$ For a line space $\left(X,\, Y,\,*\right)\,,$
the equivalence-relational join space $\left(X,\,\left\langle \cdot,\,\cdot,\,\cdot\right\rangle \right)$
is called the\emph{ join space associated with} $\left(X,\, Y,\,*\right)\,.$
Join space concepts and join space notations,
\begin{itemize}
\item \noindent when applied to a line structure on a set, refer to its
associated join relation. 
\item \noindent when applied to a line space, refer to its associated join
space.
\end{itemize}
\noindent In particular, for $a,\, c\in X\,,$
\begin{align*}
ac & =\begin{cases}
\overleftrightarrow{ac} & \mbox{if }a\neq c\\
\left\{ a\right\}  & \mbox{if }a=c
\end{cases}\,.
\end{align*}

\begin{prop}
\label{sub:equivalence_relationality_criterion} (equivalence-relationality
criterion) Let $X$ be a join space. The following conditions are
equivalent:
\begin{enumerate}
\item \label{enu:equivalence_relationality_criterion_1}$X$ is equivalence-relational.
\item \label{enu:equivalence_relationality_criterion_2}For all $a,\, b,\, c\in X\,,$
$\left\langle a,\, b,\, c\right\rangle $ and $a\neq b$ imply $ab=ac\,.$
\item \label{enu:equivalence_relationality_criterion_3}For all $a,\, b,\, c,\, d\in X\,,$
$c,\, d\in ab$ and $c\neq d$ imply $ab=cd\,.$
\end{enumerate}
\end{prop}
\begin{proof}
Step 1. (\ref{enu:equivalence_relationality_criterion_1}) $\Leftrightarrow$
(\ref{enu:equivalence_relationality_criterion_2}). (\ref{enu:equivalence_relationality_criterion_1})
says that for all $a\in X\,,$ $X$ is $a$-equivalend-relational.
By \ref{sub:base_point_equivalence_relationality_criterion} (base-point
equivalence-relationality criterion), this condition is equivalent
to (\ref{enu:equivalence_relationality_criterion_2}).

Step 2. (\ref{enu:equivalence_relationality_criterion_2}) $\Rightarrow$
(\ref{enu:equivalence_relationality_criterion_3}).

Case 2.1. $a\neq c\,.$ The assumption $c\in ab$ says $\left\langle a,\, c,\, b\right\rangle \,.$
With the assumption $a\neq c$ and (\ref{enu:equivalence_relationality_criterion_2})
it follows that $ac=ab\,.$ Therefore, it suffices to prove that $cd=ac\,.$
Substituting $ac=ab$ into the assumption $d\in ab\,,$ $d\in ac\,,$
i.e. $\left\langle a,\, d,\, c\right\rangle \,.$ Thus, $\left\langle c,\, d,\, a\right\rangle \,.$
With the assumption $c\neq d$ and (\ref{enu:equivalence_relationality_criterion_2})
it follows that $cd=ca\,.$ Consequently, $cd=ac\,.$

Case 2.2. $a=c\,.$ Substituting into the claim, for $a,\, b,\, d\in X$
it is to be proved that $a,\, d\in ab$ and $a\neq d$ imply $ab=ad\,,$
i.e. $\left\langle a,\, d,\, b\right\rangle $ and $a\neq d$ imply
$ad=ab\,.$ This claim is condition (\ref{enu:equivalence_relationality_criterion_2})
with the substitutions $b\rightarrow d\,,$ $c\rightarrow b\,.$

Step 3. (\ref{enu:equivalence_relationality_criterion_3}) $\Rightarrow$
(\ref{enu:equivalence_relationality_criterion_2}). Substituting $d=a$
in (\ref{enu:equivalence_relationality_criterion_3}), for $a,\, b,\, c\in X\,,$
$c,\, a\in ab$ and $c\neq a$ imply $ab=ca\,,$ i.e. $\left\langle a,\, c,\, b\right\rangle $
and $a\neq c$ imply $ac=ab\,.$ Interchanging $b$ and $c\,,$ (\ref{enu:equivalence_relationality_criterion_2})
follows.
\end{proof}
\noindent Let $\left(X,\,\left\langle \cdot,\,\cdot,\,\cdot\right\rangle \right)$
be a join space.

$\left(X,\,\left\langle \cdot,\,\cdot,\,\cdot\right\rangle \right)$
and $\left\langle \cdot,\,\cdot,\,\cdot\right\rangle $ are called\emph{
preprojective} iff the following conditions are satisfied:
\begin{itemize}
\item $X$ is equivalence-relational.
\item For all $a,\, b,\, c,\, d\in X\,,$ if $b=c$ or $a=d$ or $bc\cap ad\neq\emptyset\,,$
then $a=b$ or $c=d$ or $ab\cap cd\neq\emptyset\,.$
\end{itemize}
\noindent Each of the following two figures illustrates the case $bc\cap ad\neq\emptyset\,,$
$ab\cap cd\neq\emptyset\,.$

~

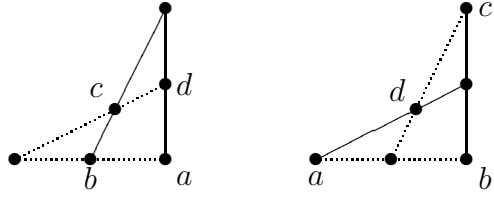
\noindent 
\[
\xy<1cm,0cm>:(2,0)*=0{\bullet}="a",(2.25,-0.25)*=0{a},(1,0)*=0{\bullet}="b",(1,-0.25)*=0{b},(2,1)*=0{\bullet}="d",(2.25,1)*=0{d},(0,0)*=0{\bullet}="e",(2,2)*=0{\bullet}="f","e";"d"**@{.}?!{"f";"b"}*{\bullet}="c",(1.08,0.92)*=0{c},"e";"a"**@{.},"f";"a"**@{-},"f";"b"**@{-},(4,0)*=0{\bullet}="a1",(4,-0.25)*=0{a},(6,0)*=0{\bullet}="b1",(6.25,-0.25)*=0{b},(6,2)*=0{\bullet}="c1",(6.25,2)*=0{c},(6,1)*=0{\bullet}="a'",(5,0)*=0{\bullet}="c'","a1";"a'"**@{-}?!{"c1";"c'"}*{\bullet}="d1",(5.08,0.92)*=0{d},"a1";"b1"**@{.},"c1";"b1"**@{-},"c1";"c'"**@{.}\endxy
\]
~

\noindent $\left(X,\,\left\langle \cdot,\,\cdot,\,\cdot\right\rangle \right)$
and $\left\langle \cdot,\,\cdot,\,\cdot\right\rangle $ are called\emph{
dense} iff it satisfies the following condition:
\begin{itemize}
\item For all $a,\, b\in X\,,$ if $a\neq b\,,$ then $ab\setminus\left\{ a,\, b\right\} \neq\emptyset\,.$
\end{itemize}
\noindent The affine join space on a left vector space is dense iff
$\left|S\right|\geq3\,.$ The line-segment join space on a left vector
space over a totally ordered division ring is dense.

$\left(X,\,\left\langle \cdot,\,\cdot,\,\cdot\right\rangle \right)$
and $\left\langle \cdot,\,\cdot,\,\cdot\right\rangle $ are called\emph{
projective} iff it is preprojective and dense. Thus, a projective
join space is a pair consisting of a set $X$ and a ternary relation
$\left\langle \cdot,\,\cdot,\,\cdot\right\rangle $ on $X$ such that
the conditions below are satisfied, with the notation
\begin{align*}
ac & :=\left\langle a,\,\cdot,\, c\right\rangle \\
 & =\left\{ x\in X\,|\,\left\langle a,\, x,\, c\right\rangle \right\} 
\end{align*}
for $a,\, c\in X$ and the numbering from \cite[§1]{pieri_1899} and
from \cite[chapter II, sect. 4]{whitehead_1906}, where a translation
of the conditions into English has been given.
\begin{itemize}
\item Postulato VII. For $a,\, c\in X\,,$ if $a\neq c\,,$ then $a\in ac\,.$
\item Postulato VI. For $a,\, c\in X\,,$ if $a\neq c\,,$ then $ac\subseteq ca\,.$
\item For $a\in X\,,$ $aa=\left\{ a\right\} \,.$
\item Postulato X. Transitivity: For all $b,\, c\in X\,,$ if $\left\langle a,\, b,\, c\right\rangle \,,$
$a\neq b$ and $a\neq c\,,$ then $ab\subseteq ac\,.$
\item Postulato IX. Symmetry: For all $a,\, b,\, c\in X\,,$ $\left\langle a,\, b,\, c\right\rangle $
and $a\neq b$ implies $\left\langle a,\, c,\, b\right\rangle \,.$
\item Postulato XII. Preprojectivity: For all $a,\, b,\, c,\, d\in X\,,$
if $b=c$ or $a=d$ or $bc\cap ad\neq\emptyset\,,$ then $a=b$ or
$c=d$ or $ab\cap cd\neq\emptyset\,.$
\item Postulato VIII. Density: For all $a,\, b\in X\,,$ if $a\neq b\,,$
then $ab\setminus\left\{ a,\, b\right\} \neq\emptyset\,.$
\end{itemize}
\noindent In \cite[§1]{pieri_1899} $ac$ is considered only for $a\neq c\,.$
The third condition, which is therefore not contained in \cite[§1]{pieri_1899},
provides a conventient convervative extension of the axiom system.
Preprojectivity does not exactly coincide with the preprojectivity
condition, but is, under the other conditions, equivalent to it.

Postulati I - V and XI from \cite[§1]{pieri_1899} are not included.
\begin{itemize}
\item Postulato I says that $X$ is a set.
\item Postulati IV, V say that for $a,\, c\in X\,,$ if $a\neq c\,,$ then
$ac$ as a subset of $X\,.$
\item Postulati II, III, XI say that the dimension of $X$ is at least $2\,.$
Here, as in the assumptions for a general projective space in \cite[§1]{veblen_young_1908},
this condition is not included in the defintion of projectivity. The
dimension of a projective join space is defined as the rank of the
associated matroid minus $1\,,$ which equals the rank if the rank
is infinite.
\end{itemize}
As noted after \ref{sub:join_space_associated_with_a_line_space}
(join space associated with a line space), join space concepts and
join space notations, when applied to a line structure on a set or
line space, refer to its associated join relation on the set or join
space, respectively. In this sense, it is well-defined when a line
space and its line structure are  called preprojective, dense, projective.
A projective line space is also called a \emph{projective space}.
Thus, a projective space is a triple $\left(X,\, Y,\,*\right)$ such
that $X,\, Y$ are sets and $*$ is a binary relation between $X$
and $Y\,,$ i.e. a subset of $X\times Y\,,$, such that the conditions
below are satisfied, with the notation
\begin{align*}
ac & :=\begin{cases}
\mbox{the unique }y\in Y\mbox{ such that }a*y\,,\, c*y & \mbox{if }a\neq c\\
\left\{ a\right\}  & \mbox{if }a=c
\end{cases}
\end{align*}
 for $a,\, c\in X$ and the numbering from \cite[§1]{veblen_young_1908}.
\begin{itemize}
\item Assumptions A1, A2. For $a,\, b\in X\,,$ if $a\neq b\,,$ then there
is exactly one $y\in Y$ such that $a*y\,,$ $b*y\,.$
\item Assumption A3. For all $a,\, b,\, c,\, d\in X\,,$ if $b=c$ or $a=d$
or $bc\cap ad\neq\emptyset\,,$ then $a=b$ or $c=d$ or $ab\cap cd\neq\emptyset\,.$
\item Assumption E. For all $a,\, b\in X\,,$ if $a\neq b\,,$ then $ab\setminus\left\{ a,\, b\right\} \neq\emptyset\,.$
\end{itemize}
Let $V$ be a left vector space. Let $X$ be the set of $1$-dimensional
subspaces of $V\,,$ $Y$ the set of $2$-dimensional subspaces of
$V$ and $*$ the relation between $X$ and $Y$ defined by
\begin{align*}
a*y & :\Leftrightarrow a\subseteq y\,.
\end{align*}

\noindent The pair $\left(Y,\,*\right)$ is a projective line structure
on $X\,.$It is not set-represented. It is called the \emph{projective
line structure over} $V\,.$ The projective line space $\left(X,\, Y,\,*\right)$
is called the \emph{projective line space over} $V$ or \emph{projective
space over} $V\,.$ The projective join relation asscociated with
the projective line structure over $V$ is called the \emph{projective
join relation over} $V\,.$ The projective join space associated with
the projective line space over $V$ is called the \emph{projective
join space over} $V\,.$

\section{The Correspondence Between Set-Represented Line Spaces and Equivalence-Relational
Join Spaces}

The following theorem establishes a natural one-to-one correspondence
between the set-represented line structures and the equivalence-relational
join relations on a set.
\begin{thm}
\label{sub:correspondence_between_set_represented_line_spaces_and_equivalence_relational_join_spaces}
(correspondence between set-represented line spaces and equivalence-relational
join spaces) Let $X$ be a set.
\begin{enumerate}
\item \label{enu:correspondence_between_set_represented_line_spaces_and_equivalence_relational_join_spaces_1}For
a set-represented line structure $L=\left(Y,\,\in\right)$ on $X\,,$
the ternary relation $\iota\left(L\right)=\left\langle \cdot,\,\cdot,\,\cdot\right\rangle _{L}$
on $X$ defined by $\left\langle a,\, b,\, c\right\rangle _{L}:\Longleftrightarrow\left(\left(a\neq c\right)\mbox{ and }b\in\overleftrightarrow{ac}\right)$
or $\left(a=c\mbox{ and }b\in\left\{ a\right\} \right)$ is an equivalence-relational
join relation on $X\,.$
\item \label{enu:correspondence_between_set_represented_line_spaces_and_equivalence_relational_join_spaces_2}Vice
versa, for an equivalence-relational join relation $\left\langle \cdot,\,\cdot,\,\cdot\right\rangle $
on $X\,,$ the pair\\
$\lambda\left(\left\langle \cdot,\,\cdot,\,\cdot\right\rangle \right):=\left(Y,\,\in\right)$
with $Y:=\left\{ ab|a,\, b\in X\,,\, a\neq b\right\} $ and $\in$
denoting set membership as usual is a set-represented line structure
on $X\,.$
\item \label{enu:correspondence_between_set_represented_line_spaces_and_equivalence_relational_join_spaces_3}$\left(\iota,\,\lambda\right)$
is an inverse pair of one-to-one correspondences between the set-represented
line structures on $X$ and the equivalence-relational join relations
on $X\,.$
\item \label{enu:correspondence_between_set_represented_line_spaces_and_equivalence_relational_join_spaces_4}Set-represented
projective line structures on $X$ correspond to projective join relations
on $X\,.$
\end{enumerate}
\end{thm}
\begin{proof}
~
\begin{enumerate}
\item This is a particular case of \ref{sub:join_space_associated_with_a_line_space}
(join space associated with a line space).
\item \noindent Step 1. Proof that for $c,\, d\in X\,,$ if $c\neq d\,,$
then there is exactly one $y\in Y$ such that $c,\, d\in y\,.$\\
Step 1.1 Existence. $y:=cd$ has the desired properties.\\
Step 1.2 Uniqueness. For $y\in Y$ it suffices to prove that $c,\, d\in y$
implies $y=cd\,.$ There are $a,\, b\in X$ such that $a\neq b$ and
$y=ab\,.$ It is to be proved that $ab=cd\,.$ Substituting $y=ab$
into the assumption $c,\, d\in y\,,$ $c,\, d\in ab\,.$ With the
assumptions that $X$ is equivalence-relational and $c\neq d$ it
follows by \ref{sub:equivalence_relationality_criterion} (equivalence-relationality
criterion) that $ab=cd\,.$\\
Step 2. Proof that for $y\in Y\,,$ there are $a,\, b\in y$ such
that $a\neq b\,.$ There are $a,\, b$ such that $a\neq b$ and $y=ab\,.$
Substituting into $a,\, b\in ab\,,$ $a,\, b\in y\,.$
\item Step 1. Proof that for each set-represented line structure $L$ on
$X\,,$ $\lambda\left(\iota\left(L\right)\right)=L\,.$
\begin{align*}
\lambda\left(\iota\left(L\right)\right) & =\lambda\left(\left\langle \cdot,\,\cdot,\,\cdot\right\rangle _{L}\right)\\
 & =\left\{ \left\langle a,\,\cdot,\, b\right\rangle _{L}|a,\, b\in X\,,\, a\neq b\right\} \\
 & =\left\{ \overleftrightarrow{ab}|a,\, b\in X\,,\, a\neq b\right\} \\
 & =L\,,
\end{align*}
where the inclusion $\supseteq$ in the last step follows from the
defining property of a line space that for $y\in L\,,$ there are
$a,\, b\in l$ such that $a\neq b\,.$\\
Step 2. Proof that for each equivalence-relational join relation $\left\langle \cdot,\,\cdot,\,\cdot\right\rangle $
on $X\,,$\\
$\iota\left(\lambda\left(\left\langle \cdot,\,\cdot,\,\cdot\right\rangle \right)\right)=\left\langle \cdot,\,\cdot,\,\cdot\right\rangle \,.$
It is to be proved that for $a,\, b,\, c\in X\,,$ $\left(a,\, b,\, c\right)\in\iota\left(\lambda\left(\left\langle \cdot,\,\cdot,\,\cdot\right\rangle \right)\right)$
iff $\left\langle a,\, b,\, c\right\rangle \,,$ i.e. for $a,\, c\in X\,,$
$\left\langle a,\,\cdot,\, c\right\rangle _{\lambda\left(\left\langle \cdot,\,\cdot,\,\cdot\right\rangle \right)}=\left\langle a,\,\cdot,\, c\right\rangle \,.$\\
Case 2.1. $a\neq c\,.$
\begin{align*}
\left\langle a,\,\cdot,\, c\right\rangle _{\lambda\left(\left\langle \cdot,\,\cdot,\,\cdot\right\rangle \right)} & =\overleftrightarrow{ac}\mbox{ in }\left(X,\,\lambda\left(\left\langle \cdot,\,\cdot,\,\cdot\right\rangle \right)\right)\\
 & =\left\langle a,\,\cdot,\, c\right\rangle \,.
\end{align*}
Case 2.2. $a=c\,.$
\begin{align*}
\left\langle a,\,\cdot,\, c\right\rangle _{\lambda\left(\left\langle \cdot,\,\cdot,\,\cdot\right\rangle \right)} & =\left\{ a\right\} \\
 & =\left\langle a,\,\cdot,\, c\right\rangle \,.
\end{align*}

\item is entailed by the definition of projectivity of a line space as projectivity
of the associated join space.
\end{enumerate}
\end{proof}
\noindent For an equivalence-relational join space $\left(X,\,\left\langle \cdot,\,\cdot,\,\cdot\right\rangle \right)\,,$
the set-represented line space $\left(X,\, Y,\,\in\right)$ with the
set of lines $Y$ as defined in \ref{sub:join_space_associated_with_a_line_space}
(\ref{enu:correspondence_between_set_represented_line_spaces_and_equivalence_relational_join_spaces_2})
(join space associated with a line space) is called the\emph{ set-represented
line space associated with} $\left(X,\,\left\langle \cdot,\,\cdot,\,\cdot\right\rangle \right)\,.$
Line space concepts and notations, when applied to an equivalence-relational
join space, refer to its associated line space.

\section{Join-Equivalence-Relationality Criterion}

Preprojective join spaces are characterized as join-equivalence-relational
join spaces, projective join spaces as dense join-equivalence-relational
join spaces. 

\begin{flushleft}
Let $\left(X,\,\left\langle \cdot,\,\cdot,\,\cdot\right\rangle \right)$
be a join space.
\par\end{flushleft}

\begin{flushleft}
A subset $C$ of $X$ is called \emph{join-closed} iff the following
equivalent conditions are satisfied:
\par\end{flushleft}
\begin{itemize}
\item $CC\subseteq C\,.$
\item For all $a,\, b\in C\,,$ $ab\subseteq C\,.$ 
\item For all $a,\, b\in C\,,$ if $a\neq b\,,$ then $ab\subseteq C\,.$
\end{itemize}
\noindent The affine join space on a left vector space over a division
ring $S$  is dense iff $\left|S\right|\geq3\,.$ The line-segment
join space on a left vector space over a totally ordered division
ring is dense.

In the affine join space on a left vector space, the join-closed sets
are the affine subspaces. In the line-segment join space on a left
vector space over a totally ordered division ring, the join-closed
sets are the convex sets.

For $A\subseteq X\,,$ the \emph{join closure} or \emph{join hull}
of $A$ in $X$ is the set
\begin{align*}
\mbox{jc}\left(A\right) & :=\bigcap\left\{ C\subseteq X|C\supseteq A\mbox{ and }C\mbox{ is join-closed.}\right\} \,.
\end{align*}

\noindent It is the smallest join-closed set in $X$ containg $A\,.$

In the affine join space on a left vector space, the join closure
of a set is its affine closure. In the line-segment join space on
a left vector space over a totally ordered division ring, the join
closure of a set is its convex closure.

$\left(X,\,\left\langle \cdot,\,\cdot,\,\cdot\right\rangle \right)$
and $\left\langle \cdot,\,\cdot,\,\cdot\right\rangle $ are called\emph{
proper} iff for all $a,\, b\in X\,,$ $ab$ is join-closed. The affine
join space on a left vector space is proper. The line-segment join
space on a left vector space over a totally ordered division ring
is proper.

$\left(X,\,\left\langle \cdot,\,\cdot,\,\cdot\right\rangle \right)$
and $\left\langle \cdot,\,\cdot,\,\cdot\right\rangle $ are called\emph{
join-transitive} iff the following condition is satisfied:
\begin{itemize}
\item \noindent For all $a,\, b\in X\,,$ $X$ is $ab$-transitive.
\end{itemize}
\noindent $\left(X,\,\left\langle \cdot,\,\cdot,\,\cdot\right\rangle \right)$
and $\left\langle \cdot,\,\cdot,\,\cdot\right\rangle $ are called\emph{
join-equivalence-relational} iff the following equivalent conditions
are satisfied:
\begin{itemize}
\item \noindent For all $a,\, b\in X\,,$ $X$ is $ab$-equivalence-relational.
\item $X$ is join-transitive and symmetric.
\end{itemize}
\noindent The affine join space on a left vector space is join-equivalence-relational.
The line-segment join space on a left vector space over a totally
ordered division ring is join-transitive but in general not join-equivalence-relational.

If $X$ is join-equivalence-relational, then it is equivalence-relational.

The following theorem is theorem 2.3 from \cite{retter_2014}, There,
earlier partial results have been cited. Condition (\ref{enu:join_transitivity_criterion_2})
is the join relation version of the strict join relation condition
Assioma XIII in \cite[§10]{peano_1889}, tanslated in \cite[chapter I, sect. 3]{whitehead_1907}.
\cite[sections 1.4, 1.5]{retter_2013} contains more examples, counter-examples,
alternative terminology, some history of these concepts and further
references. Here,  the following replacements of terminology have
been made: 'interval-convex' by 'proper', 'interval' by 'join', 'convex'
by 'join-closed'. Each of following two figures illustrates
\begin{itemize}
\item condition (\ref{enu:join_transitivity_criterion_1}): For all $a,\, b,\, c,\, x,\, y\in X\,,$
if $\left\langle ab,\, x,\, y\right\rangle $ and $\left\langle ab,\, y,\, c\right\rangle \,,$
then $\left\langle ab,\, x,\, c\right\rangle \,.$
\item condition (\ref{enu:join_transitivity_criterion_2}): For all $a,\, b,\, c,\, x\in X\,,$
if $x\in a\left(bc\right)\,,$ then $x\in\left(ab\right)c\,.$ 
\end{itemize}
~

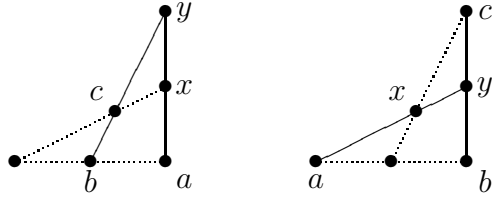
\noindent 
\[
\xy<1cm,0cm>:(2,0)*=0{\bullet}="a",(2.25,-0.25)*=0{a},(1,0)*=0{\bullet}="b",(1,-0.25)*=0{b},(2,1)*=0{\bullet}="x",(2.25,1)*=0{x},(0,0)*=0{\bullet}="e",(2,2)*=0{\bullet}="y",(2.25,2)*=0{y},"e";"x"**@{.}?!{"y";"b"}*{\bullet}="c",(1.08,0.92)*=0{c},"e";"a"**@{.},"y";"a"**@{-},"y";"b"**@{-},(4,0)*=0{\bullet}="a1",(4,-0.25)*=0{a},(6,0)*=0{\bullet}="b1",(6.25,-0.25)*=0{b},(6,2)*=0{\bullet}="c1",(6.25,2)*=0{c},(6,1)*=0{\bullet}="y1",(6.25,1)*=0{y},(5,0)*=0{\bullet}="c1'","a1";"y1"**@{-}?!{"c1";"c1'"}*{\bullet}="x",(5.08,0.92)*=0{x},"a1";"b1"**@{.},"c1";"b1"**@{-},"c1";"c1'"**@{.}\endxy
\]
~
\begin{thm}
\label{sub:join_transitivity_criterion} (join-transitivity criterion)
Let $X$ be a join space. Then the following conditions are equivalent:
\begin{enumerate}
\item \label{enu:join_transitivity_criterion_1}$X$ is join-transitive.
\item \label{enu:join_transitivity_criterion_2}For all $a,\, b,\, c\in X\,,$
$a\left(bc\right)\subseteq\left(ab\right)c\,.$
\item \label{enu:join_transitivity_criterion_3}For all $a,\, b,\, c\in X\,,$
$a\left(bc\right)=\left(ab\right)c\,.$
\item \label{enu:join_transitivity_criterion_4}$P\left(X\right)$ with
the binary operation $\left(A,\, C\right)\mapsto AC$ is a semigroup.
\item \label{enu:join_transitivity_criterion_5}$P\left(X\right)$ with
the binary operation $\left(A,\, C\right)\mapsto AC$ is a commutative
semigroup.
\item \label{enu:join_transitivity_criterion_6}$X$ is proper, and for
each join-closed set $A\,,$ the binary relation $\left\langle A,\,\cdot,\,\cdot\right\rangle $
on $X$ is transitive.
\item \label{enu:join_transitivity_criterion_7}For all join-closed sets
$A,\, C\,,$ $AC$ is join-closed.
\item \label{enu:join_transitivity_criterion_8}For all $a,\, b,\, c\in X\,,$
$\left(ab\right)c$ is join-closed.
\item \label{enu:join_transitivity_criterion_9}For all $a,\, b,\, c\in X\,,$
$\mbox{jc}\left(\left\{ a,\, b,\, c\right\} \right)=\left(ab\right)c\,.$
\end{enumerate}
\end{thm}
~
\begin{thm}
\label{sub:join_equivalence_relationality_criterion} (join-equivalence-relationality
criterion) Let $X$ be an equivalence-relational join space. The following
conditions are equivalent:
\begin{enumerate}
\item \label{enu:join_equivalence_relationality_criterion_1}$X$ is join-equivalence-relational.
\item \label{enu:join_equivalence_relationality_criterion_2}$X$ is join-transitive.
\item \label{enu:join_equivalence_relationality_criterion_3}For all $a,\, b,\, c\in X\,,$
$a\left(bc\right)\subseteq\left(ab\right)c\,.$
\item \label{enu:join_equivalence_relationality_criterion_4}For all $a,\, b,\, c,\, d\in X\,,$
if $\left(c,\, b,\, a,\, d\right)$ is dependent, then $\left(a,\, b,\, c,\, d\right)$
is dependent. 
\item \label{enu:join_equivalence_relationality_criterion_5}$X$ is preprojective.
\end{enumerate}
\end{thm}
\begin{proof}
Step 1. (\ref{enu:join_equivalence_relationality_criterion_1}) $\Leftrightarrow$
(\ref{enu:join_equivalence_relationality_criterion_2}). The assumption
that $X$ is equivalence-relational entails that $X$ is symmetric.
Consequently, (\ref{enu:join_equivalence_relationality_criterion_1}),
which says that $X$ is symmetric and join-transitive, is equivalent
to (\ref{enu:join_equivalence_relationality_criterion_2}).

Step 2. (\ref{enu:join_equivalence_relationality_criterion_2}) $\Leftrightarrow$
(\ref{enu:join_equivalence_relationality_criterion_3}) is a particular
case of \ref{sub:join_transitivity_criterion} (join-transitivity
criterion).

Step 3. (\ref{enu:join_equivalence_relationality_criterion_3}) $\Leftrightarrow$
(\ref{enu:join_equivalence_relationality_criterion_4}). It is to
be proved: For all $a,\, b,\, c,\, d\in X\,,$ $d\in a\left(bc\right)\Longrightarrow d\in\left(ab\right)c$
iff for all $a,\, b,\, c,\, d\in X$ $\left(c,\, b,\, a,\, d\right)$
dependent implies $\left(a,\, b,\, c,\, d\right)$ dependent. The
assumption that $X$ is equivalence-relational entails that $X$ is
$a$-equivalence-relational and $b$-symmetric. The claim follows
by \ref{sub:rejoinability_criterion} (rejoinability criterion).

Step 4. (\ref{enu:join_equivalence_relationality_criterion_5}) $\Leftrightarrow$
(\ref{enu:join_equivalence_relationality_criterion_4}). It is to
be proved: for all $a,\, b,\, c,\, d\in X$ $\left(c,\, b,\, a,\, d\right)$
dependent implies $\left(a,\, b,\, c,\, d\right)$ dependent. iff
For all $a,\, b,\, c,\, d\in X\,,$ if $b=c$ or $a=d$ or $bc\cap ad\neq\emptyset\,,$
then $a=b$ or $c=d$ or $ab\cap cd\neq\emptyset\,.$ The assumption
that $X$ is equivalence-relational entails that $X$ is $c$-symmetric.
The claim follows by \ref{sub:quadruple_dependence_criterion} (quadruple
dependence criterion).\end{proof}
\begin{cor}
\label{sub:projectivity_criterion} (projectivity criterion) A join
space is projective iff it is dense and join-equivalence-relational.
\end{cor}

\section{Matroid Criteria}

It is shown how the concepts  of a preprojective join space and of
a projective join space can be derived from the concept of a matroid. 

Let $X$ be a set. A \emph{closure system} or \emph{Moore family}
on $X$ is a set $C$ of subsets of $X$ such that $X\in C$ and for
each non-empty $D\subseteq C\,,$ $\bigcap D\in C\,.$

A \emph{closure space} is a pair consisting of a set $X$ and a closure
system $C$ on $X\,.$ A set $A\subseteq X$ is called \emph{closed}
iff $A\in C\,.$ When $\left(X,\, O\right)$ is a topological space,
then the pair consisting of $X$ and the set of closed sets in $\left(X,\, O\right)$
is a closure space. When $\left(X,\,\left\langle \cdot,\,\cdot,\,\cdot\right\rangle \right)$
is a join space, then the pair consisting of $X$ and the set of join-closed
sets is a closure space. When $\left(X,\,\cdot\right)$ is a group,
then the pair consisting of $X$ and the set of subgroups is a closure
space. The concept of a closure space as defined here is slighly more
general than in \cite[chapter I, 1.2]{vel_1993}, where it is required
that $\emptyset\in C$ and a closure system is called a protopology.
A group with its set of subgroups wouldn't be a closure space under
this narrower defintion.

A closure space $\left(X,\, C\right)$ is also simply denoted by $X$
when it is clear from the context whether the closure space or only
the set is meant.

Let $\left(X,\, C\right)$ be a closure space.

For $A\subseteq X\,,$ the \emph{closure} of $A$ is the set
\begin{align*}
\mbox{cl}\left(A\right) & :=\bigcap\left\{ B\subseteq X|B\supseteq A\mbox{ and }B\in C\right\} 
\end{align*}

\noindent It is the smallest closed superset of $A\,.$ When $X$
is a join space and $C$ is the system of join-closed sets in $X\,,$
then for $A\subseteq X\,,$ the closure of $A$ is the join closure
of $A\,.$

For $A\subseteq X\,,$ the \emph{entailment relation} relative to
$A$ or $A$-entailment relation is the binary relation $\vdash_{A}$
on $X$ defined by
\begin{align*}
x\vdash_{A}y & :\Leftrightarrow y\in\mbox{cl}\left(A\cup\left\{ x\right\} \right)\,.
\end{align*}

\noindent $\left(X,\, C\right)$ is called an \emph{exchange space}
iff for each closed set $A\subseteq X\,,$ one and therefore all of
the following conditions hold, which are equivalent by \cite[Proposition 3.1.1 (3)]{retter_2013}:
\begin{itemize}
\item The relation $\vdash_{A}$ is symmetric on $X\setminus A\,.$
\item The restriction $\vdash_{A}|\left(X\setminus A\right)$ is an equivalence
relation on $X\setminus A\,.$
\end{itemize}
\noindent $\left(X,\, C\right)$ is called \emph{algebraic} or \emph{combinatorial}
iff for each chain $D\subseteq C\,,$ $\bigcup D\in C\,.$ \cite[chapter I, 1.3]{vel_1993}
states the equivalence of this definition with other well-known definitions.
When $X$ is a join space and $C$ is the system of join-closed sets
in $X\,,$ then $\left(X,\, C\right)$ is a combinatorial closure
space.

A\emph{ matroid} is a combinatorial exchange space. When $\left(X,\,+,\,\cdot\right)$
is a vector space over a division ring, then the pair consisting of
$X$ and the set of subspaces is a matroid. The concept of a matroid
as defined here is slighly more general than in \cite[chapter I, 1.2]{vel_1993},
where it is, via the definition of a closure space, required that
$\emptyset\in C\,.$ A vector space with its set of subspaces wouldn't
be a matroid under this narrower defintion.

\noindent The following proposition is proposition 3.3 from \cite{retter_2014}.
\begin{prop}
\label{sub:interval_transitive_interval_spaces} (join-transitive
join spaces) Let $X$ be a join-transitive join space and $A$ a join-closed
set. Then the relative entailment relation $\vdash_{A}$ is the reverse
relation of the binary relation $\left\langle A,\,\cdot,\,\cdot\right\rangle \,.$\end{prop}
\begin{proof}
\noindent \cite[proposition 3.3]{retter_2014}
\end{proof}
\noindent The following proposition is proposition 3.4 from \cite{retter_2014}.
It is a particular case of a more general principle for relational
structures.
\begin{prop}
\label{sub:join_spaces_are_combinatorial_closure_spaces} (join spaces
are combinatorial closure spaces) Let $X$ be a join space. Then the
closure space consisting of $X$ and the set of join-closed sets is
combinatorial.\end{prop}
\begin{thm}
\label{sub:matroid_criterion_for_join_transitive_join_spaces} (matroid
criterion for join-transitive join spaces) Let $X$ be a join-transitive
join space. The following conditions are equivalent:
\begin{enumerate}
\item \label{enu:matroid_criterion_for_join_transitive_join_spaces_1}$X$
is symmetric.
\item \label{enu:matroid_criterion_for_join_transitive_join_spaces_2}$X$
is join-equivalence-relational.
\item \label{enu:matroid_criterion_for_join_transitive_join_spaces_3}The
pair consisting of $X$ and the set of join-closed sets is an exchange
space.
\item \label{enu:matroid_criterion_for_join_transitive_join_spaces_4}The
pair consisting of $X$ and the set of join-closed sets is a matroid.
\end{enumerate}
\end{thm}
\begin{proof}
Step 1. (\ref{enu:matroid_criterion_for_join_transitive_join_spaces_1})
$\Leftrightarrow$ (\ref{enu:matroid_criterion_for_join_transitive_join_spaces_2})
(\ref{enu:matroid_criterion_for_join_transitive_join_spaces_2}) says:
$X$ is symmetric and join-transitive. With the assumption that $X$
is join-transitive, this condition is equivalent to (\ref{enu:matroid_criterion_for_join_transitive_join_spaces_1}).

Step 2. (\ref{enu:matroid_criterion_for_join_transitive_join_spaces_1})
$\Leftrightarrow$ (\ref{enu:matroid_criterion_for_join_transitive_join_spaces_3}).
It suffices to prove for each join-closed set $A$ that the binary
relation $\left\langle A,\,\cdot,\,\cdot\right\rangle $ is symmetric
on $X\setminus A$ iff the relative entailment relation $\vdash_{A}$
is symmetric on $X\setminus A\,.$ Symmetry being preserved under
passing to the reverse relation, it suffices to prove that for each
join-closed set $A\,,$ the relation $\vdash_{A}$ is the reverse
relation of the relation $\left\langle A,\,\cdot,\,\cdot\right\rangle \,.$
This claim follows by \ref{sub:interval_transitive_interval_spaces}
(join-transitve join spaces) from the assumption that $X$ is join-transitive.

Step 3. (\ref{enu:matroid_criterion_for_join_transitive_join_spaces_3})
$\Leftrightarrow$ (\ref{enu:matroid_criterion_for_join_transitive_join_spaces_4})
follows by \ref{sub:join_spaces_are_combinatorial_closure_spaces}
(join spaces are combinatorial closure spaces).\end{proof}
\begin{cor}
\label{sub:matroid_preprojectivity_criterion} (matroid preprojectivity
criterion) Let $X$ be a join space. $X$ is preprojective iff it
is join-transitive and the pair consisting of $X$ and the set of
join-closed sets is a matroid.\end{cor}
\begin{proof}
By \ref{sub:projectivity_criterion} (projectivity criterion) and
\ref{sub:matroid_criterion_for_join_transitive_join_spaces} (matroid
criterion for join-transitive join spaces), the following conditions
are equivalent:
\begin{align*}
 & X\mbox{ is preprojective.}\\
 & X\mbox{ is join-equivalence-relational.}\\
 & X\mbox{ is join-transitive and symmetric.}\\
 & X\mbox{ is join-transitive and the pair consisting of }X\mbox{ and the join-closed sets is a matroid.}
\end{align*}
\end{proof}
\begin{cor}
\label{sub:matroid_projectivity_criterion} (matroid projectivity
criterion) Let $X$ be a join space. $X$ is projective iff it is
dense and join-transitive and the pair consisting of $X$ and the
set of join-closed sets is a matroid.
\end{cor}

\section{Conclusion}

In \ref{sub:correspondence_between_set_represented_line_spaces_and_equivalence_relational_join_spaces}
(\ref{enu:correspondence_between_set_represented_line_spaces_and_equivalence_relational_join_spaces_4})
(correspondence between set-represented line spaces and equivalence-relational
join spaces) provides a natural one-to-one correspondence between
projective spaces, defined by axioms A1, A2, A3, E in \cite[§1]{pieri_1899};
and projective join spaces, defined by Postulati VI; VII, VIII, IX,
X, XII in \cite[§1]{pieri_1899}. \ref{sub:join_equivalence_relationality_criterion}
(join-equivalence-relationality criterion) and \ref{sub:projectivity_criterion}
(projectivity criterion) amount to replacing the projective geometry
axiom Postulato XII in \cite[§1]{pieri_1899} by Assioma XIII in \cite[§10]{peano_1889},
where it is part of an axiom system for order geometry. Thus, projective
geometry and order geometry have a broad common axiomatic basis. As
a corollaries, \ref{sub:matroid_preprojectivity_criterion} (matroid
preprojectivity criterion and \ref{sub:matroid_projectivity_criterion}
(matroid projectivity criterion) show how the concepts of a preprojective
join space and of a projective join space can be derived from the
concept of a matroid. The defining properties of an equivalence relation
have been used as a conceptual red thread, in analogy to \cite[chapter 3]{retter_2013}
and \cite{retter_2014}, where the defining properties of a partial
order have been used as a conceptual red thread.

\end{document}